\documentclass{amsart}
\usepackage{mathabx}
\usepackage{graphicx}
\usepackage{amsfonts}
\usepackage{amsmath}
\usepackage{amscd}
\usepackage{amssymb}
\usepackage{verbatim}
\usepackage{hyperref}
\usepackage{marginnote}
\usepackage{tikz}
\usepackage{cite}
\usepackage{hyperref}
\usepackage{cleveref}
\usepackage{mathdots}

\newenvironment{proof*}{\noindent\emph{Proof}}{$\square$\smallskip}

\newtheorem{theorem}{Theorem}[section]

\newtheorem{Definition}[theorem]{Definition}
\newtheorem{lemma}[theorem]{Lemma}
\newtheorem{Example}[theorem]{Example}

\newtheorem{corollary}[theorem]{Corollary}
\newtheorem{Remark}[theorem]{Remark}
\newtheorem{proposition}[theorem]{Proposition}

\newtheorem{Exercise}[theorem]{Exercise}
\newtheorem{Exercises}[theorem]{Exercises}
\newtheorem{Notation}[theorem]{Notation}
\newtheorem{Convention}[theorem]{Convention}
\newtheorem{standing assumption}[theorem]{Standing Assumption}

\newenvironment{definition}{\begin{Definition}\normalfont}{\end{Definition}}
\newenvironment{example}{\begin{Example}\normalfont}{\end{Example}}
\newenvironment{remark}{\begin{Remark}\normalfont}{\end{Remark}}

\title[Simple expansion sets and non-positive curvature]{Simple expansion sets and non-positive curvature}  
\author[D.~S.~Farley]{Daniel S. Farley}
\address{Department of Mathematics\\ Miami University\\ Oxford, OH 45056 U.S.A.}
\email{farleyds@miamioh.edu}

\date{\today}

\begin{document}

\begin{abstract}
An \emph{expansion set} is a set $\mathcal{B}$ such that each
$b \in \mathcal{B}$ is equipped with a set of expansions $\mathcal{E}(b)$. 
The theory of expansion sets offers a systematic approach to the construction of classifying spaces for generalized Thompson groups, as described in \cite{Farley}.

We say that $\mathcal{B}$ 
is \emph{simple} if proper expansions are unique when they exist, or, equivalently, if $|\mathcal{E}(b)| \leq 2$ for all $b \in \mathcal{B}$. 

We will prove that any given simple expansion set determines a cubical complex with a metric of non-positive curvature. In many cases, the cubical complex will be CAT(0). 
We are thus able to recover proofs that Thompson's groups $F$, $T$, and $V$  \cite{MyThesis, Morethesis}, Houghton's groups $H_{n}$, and groups defined by finite similarity structures \cite{Hughes, FH1, FH2} all act on CAT(0) cubical complexes. We further state a sufficient condition for the cubical complex to be locally finite, and show that the latter condition is satisfied in the cases of $F$, $T$, $V$, and $H_{n}$.
\end{abstract}

\subjclass[2010]{Primary 20F65; Secondary 20F67}

\keywords{generalized Thompson groups, CAT(0) cubical complexes} 

\maketitle

\setcounter{tocdepth}{2} \tableofcontents

\section{Introduction}
The class of generalized Thompson groups has been widely studied in recent years. 
 Many such groups are characterized by local (or piecewise) definitions, such as Thompson's group $V$ \cite{CFP}, the Brin-Thompson groups $nV$ \cite{BrinHighD}, Stein-Thompson groups \cite{Stein}, and Nekrashevych-Roever groups \cite{Rover,Nek}, among others.  
In \cite{Farley} the author, streamlining joint work with Hughes \cite{FH2}, introduced expansion sets, which are a tool to facilitate the construction of classifying spaces for generalized Thompson groups.  Consider a generalized Thompson group $G$ acting on a space $X$. 
One takes a ``piecewise" approach to the construction of the classifying space for $G$, in the form of a choice of expansion set $\mathcal{B}$. 
Each $b \in \mathcal{B}$ has a partially ordered set $\mathcal{E}(b)$ (of \emph{expansions}) and a \emph{support} $supp(b) \subseteq X$ assigned to it. One can then form a \emph{full-support complex} $\Delta^{f}_{\mathcal{B}}$. A vertex of
$\Delta^{f}_{\mathcal{B}}$ is a finite subset $\{ b_{1}, \ldots, b_{k} \}$ of $\mathcal{B}$, where
the $b_{i}$ have disjoint support and the union of their supports is $X$. The simplicial complex structure $\Delta^{f}_{\mathcal{B}}$ is
determined completely by the sets $\mathcal{E}(b)$, which we view as simplicial complexes. The general approach of \cite{Farley} is to reduce questions about the topology of $\Delta^{f}_{\mathcal{B}}$ to questions about the (usually much smaller) sets $\mathcal{E}(b)$. We will use Thompson's group $V$  as a running example -- see Examples \ref{example:Vpt1}, \ref{example:Vpt2}, and Subsection \ref{subsection:V} for a discussion of
$V$, and the remainder of Section \ref{section:apps} for further examples.

The simplicial complexes $\Delta^{f}_{\mathcal{B}}$ are similar to the ones introduced by Stein \cite{Stein}. Indeed, \cite{Farley} and \cite{FH2} offer a general theory 
of what are sometimes known as \emph{Stein complexes}. 

In this note, we consider an especially simple type of expansion set, in which $|\mathcal{E}(b)| \leq 2$ for each $b \in \mathcal{B}$. (Thus, each simplicial complex $\mathcal{E}(b)$ is either a point or a line segment.) We call these expansion sets \emph{simple}. Our main theorem says that $\Delta^{f}_{\mathcal{B}}$ is a non-positively curved cubical complex when $\mathcal{B}$ is simple, and, in fact,
$\Delta^{f}_{\mathcal{B}}$ is a CAT(0) cubical complex under mild additional hypotheses. We can also state a sufficient condition for $\Delta^{f}_{\mathcal{B}}$ to be locally finite.

Our main theorem has several applications. For instance, we sketch  proofs that Thompson's groups $F$, $T$, and $V$ \cite{CFP}, Houghton's groups $H_{n}$ \cite{Houghton}, and groups determined by finite similarity structures \cite{Hughes} (or \emph{FSS groups}) all act on CAT(0) cubical complexes with finite stabilizers. In the cases of $F$, $T$, $V$, and $H_{n}$, the complexes are also locally finite. One suspects that the complexes for FSS groups are also locally finite in most cases, but we state no general result to that effect.

The first proofs that $F$, $T$, and $V$ act properly on CAT(0) cubical complexes were published by the author (\cite{MyThesis} for $F$, and \cite{Morethesis} for $T$ and $V$). The first proof that the Houghton groups $H_{n}$ act properly on CAT(0) cubical complexes is due to Lee \cite{Lee}. Another proof that the Houghton groups act properly on CAT(0) cubical complexes appears in \cite{FHBraided}. A common feature of the proofs in \cite{MyThesis}, \cite{Morethesis}, and \cite{FHBraided} is that all appeal to the theory of diagram groups \cite{GubaSapir}. The reasoning in this note is arguably more direct, appealing only to the existence of a simple expansion rule. A proof that FSS groups act on CAT(0) cubical complexes appeared in \cite{Hughes}, in an appendix by the author. The latter proof, however, did not give an explicit description of the cubical complex, but rather appealed to a well-known construction
by Sageev \cite{Sageev}. In this note, the construction is explicit. The complexes given here are equivalent to the ones from \cite{FH2}, but the proof that they are CAT(0) is new.


\section{Simple Expansion Sets}

In this section, we will define simple expansion sets and their associated simplicial complexes $\Delta^{f}_{\mathcal{B}}$. To define the broader class of expansion sets would take us into complications that are irrelevant to the goals of this paper. We refer the interested reader to \cite{Farley} for the complete definition. We note, however, that every simple expansion set is an expansion set in the sense of \cite{Farley}. 


\subsection{The definition of simple expansion set}

The following is slightly adapted from \cite{Farley} (Definition 2.1).

\begin{definition} \label{definition:simpleexpansionsets}
(Simple expansion sets; cf. Definition 2.1 from \cite{Farley}) 
A \emph{simple expansion set over $X$} is a $4$-tuple $(\mathcal{B}, X, supp, \mathcal{E})$, where $\mathcal{B}$ and $X$ are sets, and $supp: \mathcal{B} \rightarrow \mathcal{P}(X)$ and $\mathcal{E}$ are functions. 

For each $b \in \mathcal{B}$, $supp(b)$ is required to be a non-empty subset of $X$. The function $supp$ is called the \emph{support function}, and $supp(b)$ is the \emph{support} of $b$.  

A \emph{vertex} is a finite subset $v = \{ b_{1}, \ldots, b_{k}\} \subseteq \mathcal{B}$ such that $supp(b_{i}) \cap supp(b_{j}) = \emptyset$ when $i \neq j$. The set of all vertices is denoted $\mathcal{V}_{\mathcal{B}}$. For each $v \in \mathcal{V}_{\mathcal{B}}$, we define 
\[ supp(v) = \bigcup_{\ell =1}^{k} supp(b_{\ell}); \quad \quad  
P(v) = \{ supp(b_{\ell}) \mid \ell \in \{ 1, \ldots, k \} \}.  \]
The collection $P(v)$ is the \emph{partition induced by $v$}. It is a partition of $supp(v)$.

 The function $\mathcal{E}$ assigns a set of vertices, denoted $\mathcal{E}(b)$, to each $b \in \mathcal{B}$. 
 The sets $\mathcal{E}(b)$ are required to satisfy the following three conditions:
\begin{enumerate}
\item $|\mathcal{E}(b)| \leq 2$;
\item $ \{ b \} \in \mathcal{E}(b)$;
\item If $|\mathcal{E}(b)| = 2$ and $v \in \mathcal{E}(b) - \{ \{ b \} \}$, 
then $P(v)$ is a proper refinement of $P(\{ b \})$. 
\end{enumerate}
\end{definition}

\begin{remark}
Note that if $v \in \mathcal{E}(b)$, then $supp(v) = supp(b)$, by property (3) from Definition \ref{definition:simpleexpansionsets}.

We note also that the set of vertices $\mathcal{V}_{\mathcal{B}}$ is much larger than the vertex set of the full support complex $\Delta^{f}_{\mathcal{B}}$ to be considered later. 
\end{remark}

\begin{remark} (Comparison with the definition of expansion sets)
The above definition is essentially the definition of expansion sets from \cite{Farley} (Definition 2.1), but with the simplifying assumption that $|\mathcal{E}(b)| \leq 2$. 
We address some (apparent and real) departures from the earlier definition:
\begin{itemize}
\item In the earlier definition, $\mathcal{E}(b)$ is a partially ordered set. In the above definition, this partial order is implicit: we can simply let $\{ b \} \leq v$
if $\mathcal{E}(b) = \{ \{ b \}, v \}$. If $\mathcal{E}(b) = \{ \{ b \} \}$, then the partial order is the obvious one. These are the only two possible forms that the set $\mathcal{E}(b)$ can take.
\item Properties (2) and (3) from the above definition correspond to (1) and (2) (respectively) from Definition 2.1 of \cite{Farley}. Property (3) from Definition 2.1 of \cite{Farley} is a compatibility condition that is trivially satisfied because of our property (1). 
\item The assumption that $|\mathcal{E}(b)| \leq 2$ is indeed new. Definition 2.1 of \cite{Farley} allowed $\mathcal{E}(b)$ to have any number of elements (even infinitely many). 
\end{itemize}
\end{remark}

\begin{remark}
We will almost always refer to a (simple) expansion set as simply $\mathcal{B}$, rather than the $4$-tuple $(\mathcal{B}, X, supp, \mathcal{E})$. The same was done in \cite{Farley}.
\end{remark}

\begin{example} (The expansion set construction for Thompson's group $V$)
\label{example:Vpt1} 
 We will assume that the reader is acquainted with Thompson's group $V$ in what follows. 
The source \cite{CFP} contains a standard introduction. We note that the expansion set construction of $V$ (to be revisited here for the reader's convenience)
has already appeared in \cite{Farley} and \cite{FH2}.

Let $X = \prod_{i=1}^{\infty} \{ 0, 1 \}$. This is the Cantor set, whose members are infinite binary sequences. For a given finite binary sequence $\omega$, we let $B_{\omega}$ denote the
set of all infinite binary sequences beginning with the prefix $\omega$. We call the sets $B_{\omega}$ \emph{balls} since they are metric balls with respect to a natural ultrametric $d$.

Let $S_{V}$ denote the set of all transformations $\sigma_{\omega_{1}}^{\omega_{2}}: B_{\omega_{1}} \rightarrow B_{\omega_{2}}$
such that, for a given sequence $\omega' \in B_{\omega_{1}}$, the effect of applying 
$\sigma_{\omega_{1}}^{\omega_{2}}$ is to remove the prefix $\omega_{1}$, and replace it with 
$\omega_{2}$:
\[ \omega' = \omega_{1}a_{1}a_{2}a_{3} \ldots  \quad \mapsto \quad \omega_{2}a_{1}a_{2}a_{3}\ldots. \]
We also let $0$ denote the transformation with empty domain and codomain, and include $0$ as a member of $S_{V}$. With these conventions, $S_{V}$ is closed under compositions and inverses, making it an inverse semigroup. (Here composition is defined ``on overlaps": the domain of $f \circ g$ is the collection 
of all members of the domain of $g$ that map to members of the domain of $f$.)

Any member of $V$ is a finite disjoint union of transformations from $S_{V}$. Indeed, given 
a finite collection $C = \{ \sigma_{1}, \sigma_{2}, \ldots, \sigma_{m} \}  \subseteq S_{V}$, $C$ determines a member $g_{C}$ of $V$ if and only if the domains of the $\sigma_{i}$ partition $\mathcal{C}$, and, likewise, the codomains of the $\sigma_{i}$ partition $\mathcal{C}$. Conversely, 
every member of $V$ is equal to $g_{C}$, for some collection $C \subseteq S_{V}$. The collection $C$ is not unique (i.e., the map $C \mapsto g_{C}$ is many-to-one), for essentially the same reason that tree representatives of members  of $V$ are not unique.  

Now we can define the expansion set construction for $V$. We consider all pairs $(f,B)$, where
$B = B_{\omega}$ for some finite binary sequence $\omega$, and $f$ is a finite disjoint union of 
members of $S_{V}$ such that $B$ is the domain of $f$. (This is equivalent to $f$ being the restriction of some member of $V$ to $B$ in the case when $B$ is a proper subset of $X$.) We write $(f_{1},B_{1}) \sim (f_{2},B_{2})$ if there is a transformation
$\sigma \in S_{V}$ such that $\sigma(B_{1}) = B_{2}$, and $f_{2} \circ \sigma = f_{1}$. The relation
$\sim$ is an equivalence relation on the set of pairs $(f,B)$. We let $[f,B]$ denote the equivalence class of $(f,B)$, and let $\mathcal{B}$ denote the set of all such equivalence classes. The set $\mathcal{B}$ is the required expansion set. For a given equivalence class $b=[f,B_{\omega}]$, we let $supp(b) = f(B_{\omega})$ and 
\[ \mathcal{E}(b) = \{ \{ b \}, \{ [f_{\mid}, B_{\omega0}], [f_{\mid}, B_{\omega1}] \} \}. \]
Thus, the allowable expansions from $b$ consist of the trivial expansion $\{ b \}$, and the subdivision 
into left and right halves. It is straightforward to show that $\mathcal{E}$ and $supp$ are well-defined. Since $|\mathcal{E}(b)| = 2$ for all $b \in \mathcal{B}$, $\mathcal{B}$ is simple. 
\end{example}

\subsection{The complex $\Delta^{f}_{\mathcal{B}}$}

Next we will consider the complexes $\Delta^{f}_{\mathcal{B}}$. For the most part, we will simply summarize the relevant portions of \cite{Farley}, while recalling the definitions that we need.

\begin{definition} \label{definition:restriction} (Restriction functions; cf. Definition 2.1 from \cite{Farley})
Let $b \in \mathcal{B}$. For a given vertex $v \in \mathcal{V}_{\mathcal{B}}$, we define
\[ r_{b}(v) = \{ b' \in v \mid supp(b') \subseteq supp(b) \}. \]
This is the \emph{restriction of $v$ to the support of $b$}.
\end{definition}


\begin{definition} \label{definition:complexes}
(The full-support subcomplex $\Delta^{f}_{\mathcal{B}}$)
For a given expansion set $\mathcal{B}$, we define a simplicial complex
$\Delta^{f}_{\mathcal{B}}$ as follows:
\begin{enumerate}
\item The vertices of $\Delta^{f}_{\mathcal{B}}$ are the members $v$ of 
$\mathcal{V}_{\mathcal{B}}$ such that $supp(v) = X$;
\item A collection $\{ v_{0}, \ldots, v_{n} \}$ of vertices in $\Delta^{f}_{\mathcal{B}}$ 
span a simplex if and only if:
\begin{enumerate}
\item $|v_{i}| \neq |v_{j}|$ if $i \neq j$;
\item if $|v_{0}| < |v_{1}| < \ldots < |v_{n}|$ and $v_{0} = \{ b_{1}, \ldots, b_{m} \}$,
then 
\[ r_{b_{i}}(v_{0}) = \{ b_{i} \} \leq r_{b_{i}}(v_{1}) \leq \ldots \leq r_{b_{i}}(v_{n}) \]
is a chain in $\mathcal{E}(b_{i})$, for $i = 1, \ldots, m$.
\end{enumerate}
\end{enumerate}
\end{definition}

\begin{definition} \label{definition:ascstar}
(the height function and ascending stars)
For a vertex $v \in \left(\Delta^{f}_{\mathcal{B}}\right)^{0}$, we call $|v|$ the \emph{height of $v$}. The \emph{ascending star of $v$}, denoted $st_{\uparrow}(v)$, is the subcomplex of 
$\Delta^{f}_{\mathcal{B}}$ consisting of all simplices in which $v$ is the vertex of minimum height, and 
the faces of such simplices.
\end{definition}

\begin{definition} \label{definition:poset}
(the ascending star as a partially ordered set; cf. Definition 2.15 from \cite{Farley})
Let $v = \{ b_{1}, \ldots, b_{k} \} \in \left( \Delta^{f}_{\mathcal{B}} \right)^{0}$. We define a partial order 
on $st_{\uparrow}(v)$ as follows:
$u \leq w$ if and only if  $r_{b_{i}}(u) \leq r_{b_{i}}(w)$
for all $b_{i} \in v$, where the latter inequality is with respect to the partial order
on $\mathcal{E}(b_{i})$.
\end{definition}

\begin{theorem} \label{theorem:summary}
(Description of $st_{\uparrow}(v)$; cf. Proposition 2.16 and Theorem 2.22 from \cite{Farley})
Let $v = \{ b_{1}, \ldots, b_{k} \} \in \left( \Delta^{f}_{\mathcal{B}} \right)^{0}$. 
The ascending star $st_{\uparrow}(v)$ is abstractly isomorphic to 
$\mathcal{E}(b_{1}) \times \ldots \times \mathcal{E}(b_{k})$ as a partially ordered set. As a result, there is a homeomorphism
\[ |st_{\uparrow}(v)| \cong \prod_{i=1}^{k} |\mathcal{E}(b_{i})|. \] \qed
\end{theorem}

\begin{corollary} \label{corollary:unionofcubes}
(Union of cubes)
If $\mathcal{B}$ be a simple expansion set, then 
$\Delta^{f}_{\mathcal{B}}$ is a union of cubes. 
\end{corollary}

\begin{proof}
Definition \ref{definition:complexes} shows that each simplex
$\{ v_{0}, \ldots, v_{m} \} \subseteq \Delta^{f}_{\mathcal{B}}$ is part of the ascending star of its minimal-height vertex. Theorem \ref{theorem:summary} shows that each such ascending star is a product $|\mathcal{E}(b_{1})| \times \ldots
\times |\mathcal{E}(b_{k})|$. The definition of simple expansion set
specifies that each factor in the above product is either a point or a line segment, so the product is a cube.
\end{proof}

\section{Non-positively curved cubical complexes}

In this section, we will show that $\Delta^{f}_{\mathcal{B}}$ is a non-positively curved cubical complex when $\mathcal{B}$ is simple. Under very mild additional assumptions, we will conclude that $\Delta^{f}_{\mathcal{B}}$ is CAT(0). 

\subsection{The cubical complex condition}

The next step is to show that $\Delta^{f}_{\mathcal{B}}$ is a cubical complex.
We will need to take a closer look at the cubes from Corollary \ref{corollary:unionofcubes}.

\begin{definition} \label{definition:cubes} (Abstract cubes)
Let $v_{1}, v_{2} \in \mathcal{V}_{\mathcal{B}}$, where $v_{2} \subseteq v_{1}$
and $|\mathcal{E}(b)| = 2$ for each $b \in v_{2}$. We let $\mathcal{C}(v_{1},v_{2})$ denote the \emph{cube determined by $(v_{1},v_{2})$}.  It is the portion of 
$st_{\uparrow}(v_{1})$ corresponding to 
\[ \prod_{b' \in v_{1}-v_{2}} \{ \{ b' \} \} \times \prod_{b'' \in v_{2}} \mathcal{E}(b'')
\subseteq \prod_{b \in v_{1}} \mathcal{E}(b). \]
Thus, it is the subcomplex of $st_{\uparrow}(v_{1})$ consisting of simplices that involve expansions only at members of $v_{2}$.
\end{definition}

\begin{example} \label{example:Vpt2}
(Cubes in the case of $V$)
We consider an example of a cube in $\Delta^{f}_{\mathcal{B}}$, where $\mathcal{B}$ is the expansion set defined in Example \ref{example:Vpt1} for Thompson's group $V$.

For a pair $[id_{\mid B}, B] \in \mathcal{B}$, we will simply write $B$ for convenience. Thus, we may
 consider the set $\{ B_{0}, B_{10}, B_{11} \}$, which is a vertex in $\Delta^{f}_{\mathcal{B}}$, and the cube 
 \[ \mathcal{C} := \mathcal{C}( \{ B_{0}, B_{10}, B_{11} \}, \{ B_{0}, B_{10} \}).\]
By the definition of $\mathcal{B}$, the sets $\mathcal{E}(B_{0})$ and $\mathcal{E}(B_{10})$ are as follows:
\begin{align*}
\mathcal{E}(B_{0}) &= \{ \{ B_{0} \}, \{ B_{00}, B_{01} \} \}; \\
\mathcal{E}(B_{10}) &= \{ \{ B_{10} \}, \{ B_{100}, B_{101} \} \}. 
\end{align*}
The cube $\mathcal{C}$ is depicted in Figure \ref{figure:1}. The gray vertical line represents a $1$-simplex in $\Delta^{f}_{\mathcal{B}}$ that is not part of the cubical structure.  

\begin{figure}[!h]
\begin{center}
\begin{tikzpicture}
\filldraw[lightgray](2,0) -- (0,2) -- (2,4) -- (4,2) -- cycle;
\draw[black,thick](2,0) -- (0,2);
\draw[black,thick](2,0) -- (4,2);
\draw[black,thick](0,2) -- (2,4);
\draw[black,thick](4,2) -- (2,4);
\draw[gray](2,0) -- (2,4);
\node at (2,4.3){$\{ B_{00}, B_{01}, B_{100}, B_{101}, B_{11} \}$};
\node at (-1.7,2){$\{ B_{00}, B_{01}, B_{10}, B_{11} \}$};
\node at (2,-.3){$\{ B_{0}, B_{10}, B_{11} \}$};
\node at (5.7,2){$\{ B_{0}, B_{100}, B_{101}, B_{11} \}$};
\end{tikzpicture}
\end{center}
\caption{Here a cube in $\Delta^{f}_{\mathcal{B}}$ is pictured, where 
$\mathcal{B}$ is the expansion set for Thompson's group $V$.}
\label{figure:1} 
\end{figure}
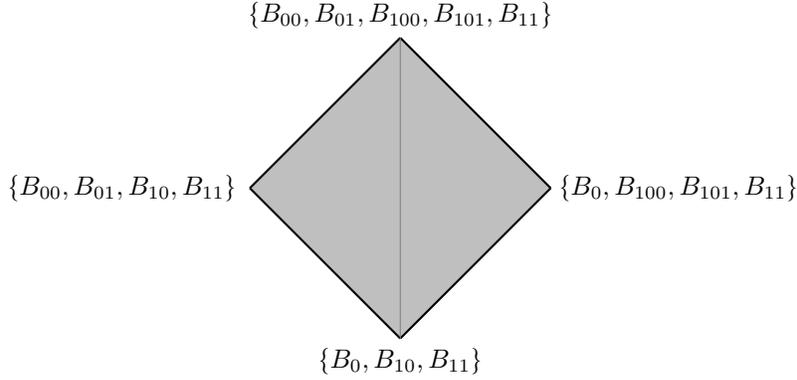
\end{example}

\begin{remark} \label{remark:cubes}
(Direct description of cubes)
Let a cube $\mathcal{C} := \mathcal{C}(v_{1}, v_{2})$ be given. It is useful to think of the members of
$v_{2}$ as \emph{active} elements of $v_{1}$, and the remaining members as \emph{inert}. The active elements have ``on'' and ``off'' positions. The ``on'' position represents an expanded state, while the ``off" position represents an unexpanded state.
If the members of $v_{2}$ are given an (arbitrary) ordering, then we can encode these ``on" and ``off" positions by binary $j$-tuples, where $|v_{2}|=j$. An ``on" position is encoded by $1$ and an ``off" position is encoded by $0$. 

For instance, in the case of Example \ref{example:Vpt2}, we could let $B_{0}$ represent the first coordinate of the $2$-tuple, and $B_{10}$ represent the second coordinate. An ``on" position 
consists of replacing the ball in question by its left and right halves, while an ``off" position leaves the ball intact. Thus, the corners of the square in Figure \ref{figure:1} would be labelled  $(1,1)$, $(0,1)$, $(0,0)$, and $(1,0)$, respectively, reading clockwise from the top vertex. 

The above description of cubes extends in an obvious way to all simple expansion sets. 
\end{remark}

\begin{definition} \label{definition:basin} (Contraction basins)
Let $\mathcal{B}$ be a simple expansion set. For a given $b \in \mathcal{B}$ such that
$|\mathcal{E}(b)|=2$, we let $Bas(b)$ be the unique member of 
$\mathcal{E}(b) - \{ \{b \} \}$. We say that $Bas(b)$ is the \emph{contraction basin of $b$}.
\end{definition}

\begin{remark}
The assignment $b \mapsto Bas(b)$ is not, in general, one-to-one. For instance, consider the example of $V$. Let $a$ denote the transformation of the Cantor set $X$ that alters only the first coordinate, changing $0$ to $1$ and $1$ to $0$. Let $b_{1} = [id_{X},X]$ and
$b_{2} = [a,X]$. By definition, $Bas(b_{1}) = \{ [id_{B_{0}}, B_{0}], [id_{B_{1}}, B_{1}] \}$
and $Bas(b_{2}) = \{ [a_{\mid B_{0}}, B_{0}], [a_{\mid B_{1}},B_{1}] \}$. Now note that
$a_{\mid B_{0}} = \sigma_{0}^{1}$ and $a_{\mid B_{1}} = \sigma_{1}^{0}$. It follows 
that 
\begin{align*}
[a_{\mid B_{0}}, B_{0}] &= [id_{B_{1}}, B_{1}]; \\
[a_{\mid B_{1}}, B_{1}] &= [id_{B_{0}}, B_{0}].
\end{align*}
since (respectively) $id_{B_{1}} \circ \sigma_{0}^{1} = a$ on $B_{0}$ 
and $id_{B_{0}} \circ \sigma_{1}^{0} = a$ on $B_{1}$. Thus, $Bas(b_{1}) = Bas(b_{2})$, although $b_{1} \neq b_{2}$. 

It is not difficult to show that $b \mapsto Bas(b)$ is, in fact, two-to-one in the case of $V$. We will offer a formal in the later discussion of $V$ (Subsection \ref{subsection:V}).
\end{remark}

\begin{lemma} (Intersection lemma) \label{lemma:intersection}
Let $\mathcal{C} = \mathcal{C}(v_{1},v_{2})$ and
$\mathcal{C}' = \mathcal{C}(v'_{1},v'_{2})$ be cubes.
If $b \in v_{1}$ and $b \in Bas(b')$ for some $b' \in v'_{2}$, 
then $b \in w$, for all $w \in \mathcal{C} \cap \mathcal{C}'$.
\end{lemma}

\begin{proof}
We assume $b \in v_{1}$ and $b \in Bas(b')$ for some $b' \in v'_{2}$. 
Let $w \in \mathcal{C} \cap \mathcal{C}'$ be arbitrary, and suppose for a contradiction that $b \not \in w$. Since $w \in \mathcal{C}'$, we must have 
either $\{ b' \} \subseteq w$ or $Bas(b') \subseteq w$. The latter is ruled out
because $b \in Bas(b') - w$, so $b' \in w$. Similarly, since $b \in v_{1}$ and $w \in \mathcal{C}$, it must be that $\{ b \} \subseteq w$ or $Bas(b) \subseteq w$. The former possibility is ruled out by hypothesis, so 
$Bas(b) \subseteq w$. 

Thus, it must be that $Bas(b)$ and $\{ b' \}$ are both subsets of $w$. However, if $b'' \in Bas(b)$, then 
\[ supp(b'') \subsetneq supp(Bas(b)) = supp(b)  \subsetneq supp(b'), \]
which gives two distinct members $b'$ and $b''$ of $w$ with overlapping support. This is a contradiction.
\end{proof}

\begin{proposition}
(Intersection of cubes is a cube)
If $\mathcal{C}(v_{1},v_{2})$ and $\mathcal{C}(v'_{1},v'_{2})$ are cubes,
and $\mathcal{C}(v_{1},v_{2}) \cap \mathcal{C}(v'_{1},v'_{2}) \neq \emptyset$, then
$\mathcal{C}(v_{1},v_{2}) \cap \mathcal{C}(v'_{1},v'_{2})$ is a cube $\mathcal{C}(v_{3},v_{4})$,
for some $v_{3}$ and $v_{4}$.
\end{proposition}

\begin{proof}
We write $\mathcal{C}$ and $\mathcal{C}'$ for $\mathcal{C}(v_{1},v_{2})$ and $\mathcal{C}(v'_{1},v'_{2})$, respectively. 
Assume that $\mathcal{C} \cap \mathcal{C}' \neq \emptyset$. Let $
\hat{v}$ be a vertex of minimal height in the intersection. We can write
\begin{align*}
\hat{v} &= (v_{1} - \tilde{v}_{2}) \cup \left( \cup_{b \in \tilde{v}_{2}} Bas(b) \right); \\
\hat{v} &= (v'_{1} - \tilde{v}'_{2}) \cup \left( \cup_{b' \in \tilde{v}'_{2}} Bas(b') \right),
\end{align*}
where $\tilde{v}_{2} \subseteq v_{2}$ and $\tilde{v}'_{2} \subseteq v'_{2}$.

We claim that 
\[ \mathcal{C} \cap \mathcal{C}' = \mathcal{C}\left( \hat{v}, \left( v_{2} - \tilde{v}_{2} \right) \cap \left( v'_{2} - \tilde{v}'_{2} \right)\right). \]
Indeed, we will now prove the reverse inclusion.
It follows directly from the definition of cubes that
\[ \mathcal{C}\left( \hat{v}, \left( v_{2} - \tilde{v}_{2} \right) \cap \left( v'_{2} - \tilde{v}'_{2} \right)\right) \subseteq \mathcal{C}\left( \hat{v}, \left( v_{2} - \tilde{v}_{2} \right)\right).\] Next, we note that each member of the latter cube is a member of $\mathcal{C}$, since any vertex that may obtained  from $\hat{v}$ by expanding some subset of $w' \subseteq v_{2} - \tilde{v}_{2}$ may be 
obtained from $v_{1}$ by expanding some subset of $v_{2}$: starting with $v_{1}$, one first expands each member of $\tilde{v}_{2}$ (resulting in $\hat{v}$), and then expands $w'$. It follows that
\[ \mathcal{C}\left( \hat{v}, \left( v_{2} - \tilde{v}_{2} \right)\right) \subseteq \mathcal{C}. \]
One similarly shows that
\[ \mathcal{C}\left( \hat{v}, \left( v_{2} - \tilde{v}_{2} \right) \cap \left( v'_{2} - \tilde{v}'_{2} \right)\right) 
\subseteq \mathcal{C}\left( \hat{v}, \left( v'_{2} - \tilde{v}'_{2} \right)\right) \subseteq \mathcal{C}', \]
proving the reverse inclusion, as desired.

Let us now establish the forward inclusion. First, note that $\tilde{v}_{2} \cap \tilde{v}'_{2} = \emptyset$. Indeed, if $b'' \in 
\tilde{v}_{2} \cap \tilde{v}'_{2}$, we could write
\begin{align*}
w &= \left( v_{1} - \tilde{v}_{2} \right) \cup \{ b'' \} \cup \left( \cup_{b \in \tilde{v}_{2} - \{ b'' \}} Bas(b) \right); \\
w &= \left(v'_{1} - \tilde{v}'_{2} \right) \cup \{ b'' \} \cup \left( \cup_{b' \in \tilde{v}'_{2} - \{ b'' \}} Bas(b') \right),
\end{align*}
which exhibits a vertex $w \in \mathcal{C} \cap \mathcal{C}'$ such that $|w| < |\hat{v}|$. This contradicts the minimality of $|\hat{v}|$. 

Now let $u$ be an arbitrary vertex of $\mathcal{C} \cap \mathcal{C}'$. We 
can write
\begin{align*}
u &= (v_{1} - v_{2}) \cup B_{1} \cup \ldots \cup B_{k}, \\
u &= (v'_{1} - v'_{2}) \cup B'_{1} \cup \ldots \cup B'_{\ell}
\end{align*}
where, in the first equation,
 $v_{2} = \{ b_{1}, \ldots, b_{k} \}$ and 
each $B_{i} \in \mathcal{E}(b_{i})$ (i.e., $B_{i} = \{ b_{i} \}$
or $Bas(b_{i})$).
Similarly, in the second equation, 
$v'_{2} = \{ b'_{1}, \ldots, b'_{\ell} \}$ and 
each $B'_{i} \in \mathcal{E}(b'_{i})$.

We would now like to show that 
\[ E_{\mathcal{C},u} := \{ b_{i} \in v_{2} \mid B_{i} = Bas(b_{i}) \} \supseteq \tilde{v}_{2}. \]
Indeed, suppose for a contradiction that $b \in \tilde{v}_{2} - E_{\mathcal{C},u}$. Since $b \in \tilde{v}_{2}$, we must have $Bas(b) \subseteq \hat{v}$. In particular, $b \not \in \hat{v}$. However, since $b \in v_{2} - E_{\mathcal{C},u}$,
it must be that $\{ b \} \subseteq u$, so $b \in u$. Also, Lemma \ref{lemma:intersection} and the facts that $b \not \in \hat{v}$ and $b \in v_{2} \subseteq v_{1}$ imply that $b \not \in Bas(b')$, for any $b' \in v'_{2}$. 

Consider now the second representation of $u$, given above. Since $b \in u$, it must be that $b \in v'_{1} - v'_{2}$, or 
$b \in B'_{j}$, for some $j$. The first possibility is ruled out by the fact that $v'_{1} - v'_{2} \subseteq v'_{1} - \tilde{v}'_{2}$, and, thus, were $b \in v'_{1} - v'_{2}$, we would have $b \in \hat{v}$, which we know is not true. Thus, $b \in B'_{j}$ for some $j$. It follows that $B'_{j} = \{ b'_{j} \}$, since the possibility $B'_{j} = Bas(b'_{j})$ is ruled out by the fact that $b \not \in Bas(b'_{j})$. Thus $b = b'_{j}$; i.e., $b \in v'_{2} \subseteq v'_{1}$. Furthermore, since $b \in \tilde{v}_{2}$, 
$b \not \in \tilde{v}'_{2}$. It follows that $b \in v'_{1} - \tilde{v}'_{2} \subseteq \hat{v}$. Thus, $b \in \hat{v}$, a contradiction. This proves that $\tilde{v}_{2} \subseteq E_{\mathcal{C},u}$, as desired.  

It follows that any given vertex $u \in \mathcal{C} \cap \mathcal{C}'$ may be obtained from $\hat{v}$ by some combination of expansions from $v_{2} - \tilde{v}_{2}$. 
A similar line of reasoning shows that the same $u$ may be obtained from 
$\hat{v}$ by some combination of expansions  from $v'_{2} - \tilde{v}'_{2}$. However, the expansions from $\hat{v}$ that lead to $u$ are unique, and so all must occur from the
set $(v_{2}-\tilde{v}_{2}) \cap (v'_{2} - \tilde{v}'_{2})$. 
\end{proof}

\begin{corollary}
($\Delta^{f}_{\mathcal{B}}$ is a cubical complex)
The complex $\Delta^{f}_{\mathcal{B}}$ is a cubical complex, with cubes $\mathcal{C}(v_{1},v_{2})$,
where $(v_{1},v_{2})$ runs over all pairs satisfying the conditions from Definition \ref{definition:cubes}. \qed
\end{corollary}

\subsection{The link condition}

We will now verify that Gromov's well-known link condition is satisfied by the complexes $\Delta^{f}_{\mathcal{B}}$. We assume that the reader is acquainted with links in cubical complexes and with the link condition. Bridson and Haefliger \cite{BH} is a standard source for both.

\begin{proposition} (Non-positive curvature) \label{proposition:nonpos}
The complex $\Delta^{f}_{\mathcal{B}}$, endowed with the cubes $\mathcal{C}$ from Definition
\ref{definition:cubes}, is a cubical complex satisfying Gromov's link condition.
\end{proposition}

\begin{proof}
Let $v \in \left(\Delta^{f}_{\mathcal{B}}\right)^{0}$. Let us assume that $v_{1}, \ldots, v_{m}$ are vertices in $lk(v)$ that
are pairwise connected by edges. Since each $v_{j}$ is connected to $v$ by an edge, it must be that each $v_{j}$ is the result of replacing a single $b_{j} \in v$ by 
$Bas(b_{j})$, or there is some $\hat{b} \in \mathcal{B}$ such that $v_{j}$ is the result of replacing the contraction basin $Bas(\hat{b}) \subseteq v$ by 
$\hat{b}$. In the former case, we call $\{ b_{j} \}$ an \emph{expansion basin} for $v_{j}$ relative to $v$. In the latter case, we say that $Bas(\hat{b})$ is the \emph{contraction basin} for
$v_{j}$ relative to $v$. Collectively, we refer to these as \emph{basins} for the $v_{j}$ relative to $v$. 

Next, we note that the basins for the $v_{j}$ are pairwise disjoint.
Indeed, consider two vertices $v_{i}$ and $v_{j}$ in $lk(v)$ that, by hypothesis, are connected by an edge in $lk(v)$. It follows that $v_{i}$, $v_{j}$, and $v$ are corners of a two-cube $\mathcal{C}(v',v'')$, where $v_{i}$ and $v_{j}$ are adjacent to $v$. 
Since $\mathcal{C}(v',v'')$ is two-dimensional, it must be that  $|v''| = 2$. We can let
\[ v' = \{ b'_{1}, \ldots, b'_{k} \}, \]
where $v'' = \{ b'_{1}, b'_{2} \}$, without loss of generality. Now note that 
\[ v = \{ b'_{3}, \ldots, b'_{k} \} \cup B'_{1} \cup B'_{2}, \]
where $B'_{i}$ is either $\{ b'_{i} \}$ or $Bas(b'_{i})$, for $i=1,2$. The sets $B'_{i}$ are therefore disjoint (for instance, because their supports are disjoint). By the definition of cubes, $v_{i}$ and $v_{j}$ are obtained from $v$ by expanding or contracting at one or the other of $B'_{1}$ and $B'_{2}$. If $B'_{i} = \{ b'_{i} \}$, then the operation in question is an expansion, while if
$B'_{i} = Bas(b'_{i})$, then only the contraction is possible. We note, therefore, that the basins for $v_{i}$ and $v_{j}$ are the sets $B'_{1}$ and $B'_{2}$. It follows that the basins for $v_{i}$ and $v_{j}$ are disjoint.

Finally, we consider the cube $\mathcal{C}(u_{1},u_{2})$, where $u_{1}$ is the result of simultaneously applying contractions to the contraction basins of the various $v_{j}$,
and $u_{2}$ is the subset of $u_{1}$ consisting of all of the various $b$ at the bottoms of the contraction basins for the $v_{j}$ and all of the $b$ in expansion basins. The cube
$\mathcal{C}(u_{1},u_{2})$ represents an $(m-1)$-simplex spanned by 
$\{ v_{1}, \ldots, v_{m} \}$ in $lk(v)$.
\end{proof}

\subsection{A sufficient condition for CAT(0)}

We next turn to the sufficient condition for $\Delta^{f}_{\mathcal{B}}$ to be a CAT(0) cubical complex.

\begin{definition} (Ascending path partial order) \label{definition:ascending}
Let $v_{1}, v_{2} \in \left(\Delta^{f}_{\mathcal{B}}\right)^{0}$. An 
\emph{ascending path from $v_{1}$ to $v_{2}$} is an edge-path
\[ e_{0}, \ldots, e_{m} \]
such that $\iota(e_{0}) = v_{1}$, $\tau(e_{m}) = v_{2}$, and, for $i=0,1, \ldots, m$,
$|\iota(e_{i})| < |\tau(e_{i})|$.

We write $v_{1} \preceq v_{2}$ if there is an ascending path from $v_{1}$ to $v_{2}$.
It is clear that $\preceq$ is a partial order on the vertices of $\Delta^{f}_{\mathcal{B}}$.
\end{definition}

\begin{definition} (Directed set) \label{definition:directed}
A partially ordered set $(P,\leq)$ is a \emph{directed set}
if any two members of $P$ have a common upper bound with respect to $\leq$.
\end{definition}

\begin{definition} (Relative ascending links) \label{definition:relative}
Let $u', u'' \in \mathcal{V}_{\mathcal{B}}$. We write 
$u' \preceq u''$ if there is a sequence 
\[ u' = v_{0}, v_{1}, v_{2}, \ldots, v_{m} = u'' \]
such that, for $i=1, \ldots, m$, $v_{i}$ may be obtained from $v_{i-1}$ by replacing some
$b' \in v_{i-1}$ with $Bas(b')$; i.e., 
\[ v_{i} = \left( v_{i-1} - \{ b' \} \right) \cup Bas(b'), \]
for some $b' \in v_{i-1}$. 

Let $b \in \mathcal{B}$. For $v \in \mathcal{V}_{\mathcal{B}}$, we let
 $lk_{\uparrow}(\{ b \}, v)$ be the subcomplex of $\mathcal{E}(b)$ spanned
 by
 \[ \{ u \in \mathcal{E}(b) \mid u \neq \{ b \} \text { and }
u \preceq v \}. \]
The complex $lk_{\uparrow}(\{ b \}, v)$ is called the \emph{ascending link of $b$ relative to $v$}.
\end{definition}

\begin{remark} \label{remark:defremark}
The definition of $\preceq$ in Definition \ref{definition:relative}
simply extends the one from Definition \ref{definition:ascending}. The sole difference is that, in Definition \ref{definition:relative}, we are not requiring the vertices in question to be \emph{full support} vertices; i.e., members of $\left(\Delta^{f}_{\mathcal{B}}\right)^{0}$.
\end{remark}

\begin{remark} \label{remark:ascending}
(Relative ascending links in the simple case)
We claim that, when $\mathcal{B}$ is simple, the relative ascending link
$lk_{\uparrow}(\{ b \}, v)$ is either empty or a point, and the latter holds exactly when
$\{ b \} \precneq v$. Indeed, suppose that $\{ b \} \precneq v$. We can find a sequence
\[ \{ b \} = v_{0}, v_{1}, \ldots, v_{m} = v, \]
where $m \geq 1$ (since $\{ b \} \neq v$). It follows from Definition \ref{definition:relative} that $v_{1} = Bas(b)$, and therefore $Bas(b) \preceq v$. Thus,
$lk_{\uparrow}(\{ b \}, v)$ is the simplicial complex with the single vertex $Bas(b)$.
If $\{ b \} = v$, then it follows directly that $lk_{\uparrow}(\{ b \}, v)$ is empty. 
If $\{ b \} \not \preceq v$, then a sequence of the above kind cannot exist, so there is also 
no sequence of the form
\[ Bas(b) = v_{0}, v_{1}, \ldots, v_{m} = v. \]
Thus, $lk_{\uparrow}(\{ b \}, v)$ is again empty. This proves the claim.
\end{remark}

\begin{theorem} \label{theorem:main} (CAT(0) cubical complexes)
Let $\mathcal{B}$ be a simple expansion set. If the vertices of $\Delta^{f}_{\mathcal{B}}$
are a directed set with respect to $\preceq$, then $\Delta^{f}_{\mathcal{B}}$ is 
a CAT(0) cubical complex with respect to the cubes from Definition \ref{definition:cubes}.
\end{theorem}

\begin{proof}
By Proposition \ref{proposition:nonpos}, $\Delta^{f}_{\mathcal{B}}$ has non-positive curvature.
It now suffices, by the Cartan-Hadamard Theorem \cite{BH}, to prove that $\Delta^{f}_{\mathcal{B}}$
is simply connected. 

Here we can appeal to the results of \cite{Farley}. We note that the relative ascending link
$lk_{\uparrow}(\{ b \}, v)$ is always a point when $\{ b \} \precneq v$, by Remark \ref{remark:ascending}.
It follows that $\mathcal{B}$ is an $n$-connected expansion set for all $n$, by Definition 2.30 from \cite{Farley}. Thus, by Theorem 2.32 from \cite{Farley}, $\Delta^{f}_{\mathcal{B}}$ is contractible.
\end{proof}

\begin{proposition} (Local finiteness)
\label{proposition:localfin}
Assume the hypotheses of Theorem \ref{theorem:main}. Suppose that, for each 
vertex $v \in \mathcal{V}_{\mathcal{B}}$, $v = Bas(b)$ for only finitely many $b \in \mathcal{B}$.
Then $\Delta^{f}_{\mathcal{B}}$ is locally finite.
\end{proposition}

\begin{proof}
It suffices to show that each $v \in \left( \Delta^{f}_{\mathcal{B}} \right)$ is adjacent to only finitely many vertices. Now we note that any such adjacent vertex is either the result of expanding 
at some $b \in v$, or contracting at some $Bas(b) \subseteq v$. These operations result in 
either a vertex of greater height, or lesser height, respectively. It is not possible for two vertices of the same height to be adjacent.

Since $\mathcal{B}$ is simple, there are no more than $|v|$ adjacent vertices of greater height. This is because there is at most one expansion to be made at each $b \in v$. On the other hand, the hypothesis guarantees that there are only finitely many possible contractions, making  the number of adjacent vertices of lesser height finite, as well.

\end{proof}

\section{Applications} \label{section:apps}

\subsection{Thompson's group $V$ (conclusion)} \label{subsection:V}

We will now conclude the proof that $V$ acts on a locally finite CAT(0) cubical complex. 

We must check the hypotheses of Theorem \ref{theorem:main}. We have seen that $\mathcal{B}$ is a simple expansion set (Example \ref{example:Vpt1}). Thus, it remains to show that the vertices
of $\Delta^{f}_{\mathcal{B}}$ are a directed set with respect to $\preceq$. 
Let $v = \{ b_{1}, \ldots, b_{m} \} \in \left(\Delta^{f}_{\mathcal{B}}\right)^{0}$ be arbitrary. If $e$ is an ascending edge leading upwards from $v$, then the terminal vertex of $e$ is the result of replacing 
some or all of the $b_{i}$ with their left and right halves (as described in Example \ref{example:Vpt1}).  
It follows that there is an ascending edge path that leads from $v$ up to a vertex
$v' = \{ b'_{1}, \ldots, b'_{\ell} \}$, where each $b'_{j} = [\sigma,B]$ for some 
$\sigma \in S_{V}$ such that $B$ is the domain of $\sigma$. One concludes from the definition of the equivalence relation and the definition of $S_{V}$ that $[\sigma,B] = [id_{\mid}, \sigma(B)]$ and $\sigma(B) = B_{\omega}$, for some finite binary string $\omega$. Thus, we have 
\[ v = \{ b_{1}, \ldots, b_{m} \} \preceq \{ [id, B_{\omega_{1}}], \ldots, [id, B_{\omega_{\ell}}] \} = v', \]
for some collection $B_{\omega_{1}}, \ldots, B_{\omega_{\ell}}$. 
The vertex $v'$ can be identified with a partition $\mathcal{P}'$ of $\mathcal{C}$ into the sets $B_{\omega_{i}}$.  Any refinement $\mathcal{P}''$ of $\mathcal{P}'$ by balls
can be obtained by following another ascending edge path.

Now, assembling the argument from the above observations, given any two vertices $v_{1}$ and $v_{2}$, we can write $v_{1} \preceq v'_{1}$ and
$v_{2} \preceq v'_{2}$, where $v'_{i}$ (for $i=1,2$) is a partition of $\mathcal{C}$ into balls. Any two such partitions have a common refinement, so we can find $v''$ such that $v'_{1} \preceq v''$ and
$v'_{2} \preceq v''$. Thus $v''$ is the required common upper bound of $v_{1}$ and $v_{2}$.
We  conclude from Theorem \ref{theorem:main} that $\Delta^{f}_{\mathcal{B}}$ is a CAT(0) cubical complex. 

Next we would like to check the hypothesis of
Proposition \ref{proposition:localfin}. We will in fact show that any two-element vertex $\{ b_{1}, b_{2} \} \in \mathcal{V}_{\mathcal{B}}$
is the contracting basin $Bas(b)$ for precisely two $b \in \mathcal{B}$.  (Of course, if $v \in \mathcal{V}_{\mathcal{B}}$ and $|v| \neq 2$, then $v$ cannot be the contracting basin of any $b \in \mathcal{B}$, by the definition of $\mathcal{E}$.) Indeed, let
$b_{1} = [f_{1},B_{\omega_{1}}]$ and $b_{2} = [f_{2},B_{\omega_{2}}]$. Consider 
$b_{(i,j)} = [f,X]$, where 
\[ f_{(i,j)}(x) = \begin{cases} f_{i} \sigma_{0}^{\omega_{i}}(x) & \text{ if } x \in B_{0}  \\
f_{j} \sigma_{1}^{\omega_{j}}(x)  & \text{ if } x \in B_{1} \\
\end{cases} \]
and $(i,j) \in \{ (1,2), (2,1) \}$. For both choices of $(i,j)$, $Bas(b_{(i,j)}) = \{ b_{1}, b_{2} \}$. Thus, there are at least two $b$ such that $Bas(b) = \{ b_{1}, b_{2} \}$.

Conversely, suppose that $Bas(b) = \{ b_{1}, b_{2} \}$. We can assume that $b = [f,X]$, for some
$f$ that is a finite disjoint union of members of $S_{V}$. (If $b = [g,B_{\omega}]$, then we consider the pair $[g \sigma_{\epsilon}^{\omega}, X]$, where $\epsilon$ is the empty binary string. The definition of equivalence implies $(g, B_{\omega}) \sim (g \sigma_{\epsilon}^{\omega}, X)$,
so $b = [g \sigma_{\epsilon}^{\omega}, X]$.) By our assumption, we have
\[ \{ [f_{\mid B_{0}}, B_{0}], [f_{\mid B_{1}}, B_{1}] \} = \{ b_{1}, b_{2} \}. \]
Let us assume that $[f_{\mid B_{0}}, B_{0}] = b_{1}$ and $[f_{\mid B_{1}}, B_{1}] = b_{2}$, the other case being similar. By the definition of $\sim$, we 
have $f_{\mid B_{0}} = f_{1} \sigma$, for $\sigma \in S_{V}$ that carries $B_{0}$ bijectively to $B_{\omega_{1}}$. There is only one such $\sigma$ in $S_{V}$, namely $\sigma_{0}^{\omega_{1}}$. Thus, $f_{\mid B_{0}} = f_{1} \sigma_{0}^{\omega_{1}}$. Similarly,
$f_{\mid B_{1}} = f_{2} \sigma_{1}^{\omega_{2}}$. Thus, $b = b_{(1,2)}$. It follows easily that
there are no more than two members of $\mathcal{B}$ having the contracting basin $\{ b_{1}, b_{2} \}$. Proposition \ref{proposition:localfin} now applies, showing that $\Delta^{f}_{\mathcal{B}}$ is locally finite.

The group $V$ acts on $\mathcal{B}$ on the left in a natural way, that, furthermore, permutes cubes. Thus, $V$ acts isometrically on the CAT(0) cubical complex $\Delta^{f}_{\mathcal{B}}$. It is straightforward to show that the action has finite stabilizers.
A complete proof of these facts has appeared elsewhere \cite{Farley, FH2}.

The above method of proof, with minor changes, can be used to show that $F$, $T$, and the $n$-ary generalizations $F_{n}$, $T_{n}$, $V_{n}$ act properly on  
locally finite CAT(0) cubical complexes. 

\subsection{Finite similarity structure groups}

Hughes \cite{Hughes} defined a class of groups determined by finite similarity structures. Thompson's group $V$ is a member of the class, as are the $n$-ary generalizations $V_{n}$. Most of the basic definitions in this subsection are due to Hughes \cite{Hughes}. The sole exception is the definition of the complex $\Delta^{f}_{\mathcal{B}}$, which previously appeared in \cite{Farley} and \cite{FH2}. (The complexes considered in \cite{FH1} had the same vertex sets, but a different simplicial structure.) 

 Recall that an \emph{ultrametric} on a set $X$ is a metric satisfying the following strong form of the triangle inequality:
\[ d(x,y) \leq \mathrm{max}\{ d(x,z), d(y,z) \}, \]
for all $x$, $y$, $z \in X$. Assume that $X$ is a compact ultrametric space. A \emph{finite similarity structure} assigns to each pair of metric balls $B_{1}$, $B_{2} \subseteq X$ a (possibly empty) collection
$\mathrm{Sim}_{X}(B_{1}, B_{2})$ of surjective \emph{similarities}; i.e., a collection of surjective maps
$h: B_{1} \rightarrow B_{2}$ such that $d(h(x),h(y)) = \lambda d(x,y)$, for some constant $\lambda$ that depends only on $h$. Thus, each $h$ stretches or compresses distances by a scaling factor $\lambda$. The assignments are further required to satisfy:
\begin{enumerate}
\item (finiteness) each $\mathrm{Sim}_{X}(B_{1},B_{2})$ is finite;
\item (identities) $id_{B} \in \mathrm{Sim}_{X}(B,B)$, for each ball $B$;
\item (inverses) if $h \in \mathrm{Sim}_{X}(B_{1}, B_{2})$, then $h^{-1} \in \mathrm{Sim}_{X}(B_{2},B_{1})$;
\item (compositions) if $h_{1} \in \mathrm{Sim}_{X}(B_{1},B_{2})$ and 
$h_{2} \in \mathrm{Sim}_{X}(B_{2},B_{3})$,
then $h_{2} \circ h_{1} \in \mathrm{Sim}_{X}(B_{1},B_{3})$;
\item (restrictions) if $h \in \mathrm{Sim}_{X}(B_{1}, B_{2})$ and $B$ is contained in $B_{1}$, then
$h_{\mid B} \in \mathrm{Sim}_{X}(B, h(B))$.
\end{enumerate}
Let $\gamma$ be a self-bijection of $X$.  We say that $\gamma$ is \emph{locally determined by $\mathrm{Sim}_{X}$} if there is a finite partition $\mathcal{P} = \{ B_{1}, \ldots, B_{m} \}$  of $X$ by balls such that
$\gamma(B_{i})$ is a ball and
$\gamma_{\mid B_{i}} \in \mathrm{Sim}_{X}(B_{i}, \gamma(B_{i}))$ for each $i$. It is not difficult to show that the collection of self-bijections of $X$ that are locally determined by $\mathrm{Sim}_{X}$ is a group, called the \emph{finite similarity structure (FSS) group determined by $\mathrm{Sim}_{X}$}, and 
denoted $\Gamma(\mathrm{Sim}_{X})$. It is a subgroup of $\mathrm{Homeo}(X)$. 

For each FSS group $\Gamma(\mathrm{Sim}_{X})$, we can build an expansion set $\mathcal{B}$ as follows. We begin with the set of all pairs $(f,B)$, where $B$ is a metric ball and $f$ is a \emph{local similarity embedding}. That is, $f: B \rightarrow X$ is an injection, and there is a partition
$\mathcal{P} = \{ B_{1}, \ldots, B_{m} \}$ of $B$ into balls such that $f(B_{i})$ is a ball
and $f_{\mid B_{i}} \in \mathrm{Sim}_{X}(B_{i}, f(B_{i}))$ for each $i$. Two such pairs $(f_{1},B_{1})$,
$(f_{2},B_{2})$ are equivalent if there is some $h \in \mathrm{Sim}_{X}(B_{1},B_{2})$
such that $f_{2} \circ h = f_{1}$. We write $(f_{1},B_{1}) \sim (f_{2},B_{2})$. It is straightforward to check that $\sim$ is an equivalence relation on the set of all pairs $(f,B)$. We let $[f,B]$ denote the equivalence class of $(f,B)$, and let $\mathcal{B}$ denote the set of all such equivalence classes.

The set $\mathcal{B}$ is the desired expansion set. For each $b = [f,B] \in \mathcal{B}$, we let
$supp(b) = f(B)$ and 
\[ \mathcal{E}(b) = \{ \{ b \}, \{ [f_{\mid B'},B'] \mid B' \in \mathcal{P}_{max} \}\}, \]
where  $\mathcal{P}_{max}$ is the unique partition of $B$ by maximal proper subballs. It is straightforward to check that the above assignments are well-defined, and result in a simple expansion set.

The proof that the vertices of $\Delta^{f}_{\mathcal{B}}$ form a directed set follows the general pattern set by Thompson's group $V$: each partition 
$\mathcal{P}$ of a given ball $B$ into subballs can be produced as the end of a sequence
\[ \{ B \} = \mathcal{P}_{0}, \mathcal{P}_{1}, \ldots, \mathcal{P}_{\ell} = \mathcal{P}, \]
where, for each $i \in \{0, \ldots, \ell-1 \}$, the partition $\mathcal{P}_{i+1}$ can be obtained by replacing some ball $B'$ in $\mathcal{P}_{i}$ by the maximal proper subballs of $B'$. For a given vertex $v \in \Delta^{f}_{\mathcal{B}}$ and $[f,B] \in v$, we can therefore perform expansions
on $\{ [f,B] \}$ and produce a set
\[ S = \{ [h_{1},B_{1}], \ldots, [h_{\ell},B_{\ell}] \} \]
such that $h_{i}(B_{i})$ is a ball and
$h_{i} \in \mathrm{Sim}_{X}(B_{i}, h_{i}(B_{i}))$ for all $i$.
(Here the $B_{i}$ can be any collection that witness the fact that 
$f$ is a local similarity embedding.) It then follows that 
$[h_{i}, B_{i}] = [id, h_{i}(B_{i})]$, for each $i$, so $S$ can effectively be regarded as a partition of $supp([f,B])$. Thus, applying the above reasoning entry by entry to $v$, there is a sequence of ascending edges connecting $v$ to a vertex $v'$, where the latter
vertex can be identified with a partition of $X$. The proof concludes as before: given $v_{1}, v_{2} \in \left( \Delta^{f}_{\mathcal{B}} \right)^{0}$, we can produce $v'_{1}$ and $v'_{2}$ such that 
$v_{i} \preceq v'_{i}$ and $v'_{i}$ represents a partition of $X$, for $i=1,2$. Since $v'_{1}$ and $v'_{2}$ have a common refinement $v''$,
and $v'_{i} \preceq v''$, for $i=1,2$, the vertices of $\Delta^{f}_{\mathcal{B}}$ are a directed set, as required.

The group
$\Gamma(\mathrm{Sim}_{X})$ acts on $\Delta^{f}_{\mathcal{B}}$ with finite vertex stabilizers \cite{Farley, FH2}. The author does not know a general criterion that will guarantee local finiteness of
$\Delta^{f}_{\mathcal{B}}$.

\subsection{Houghton's groups $H_{n}$}
Finally, we consider Houghton's groups $H_{n}$. These are in fact a special type of FSS group (by results of \cite{FHBraided}), but, since the latter fact is far from obvious, we include some additional detail.
 
Let $n$ be a positive integer. We let $X = \mathbb{N}_{1} \cup \ldots \cup \mathbb{N}_{n}$
be the disjoint union of $n$ copies of the set of natural numbers $\mathbb{N}$. We define transformations 
as follows:
\begin{itemize}
\item For each $i \in \{ 1, \ldots, n \}$ and $k,\ell \geq 1$, let
$\tau^{i}_{k,\ell}: [k, \infty) \rightarrow [\ell, \infty)$ be the unique translation that 
sends $[k, \infty) \subseteq \mathbb{N}_{i}$ bijectively to $[\ell, \infty) \subseteq \mathbb{N}_{i}$.
\item for each pair of points $x, y \in X$, let $\sigma_{x}^{y}$ be the unique function
from $\{ x \}$ to $\{ y \}$.
\end{itemize}
The above transformations, along with the null (empty) transformation $0$, form an inverse semigroup, which we denote $S_{H_{n}}$. The Houghton group $H_{n}$ can be defined as the collection of all bijections of $X$ that are realizable as finite disjoint unions of members of $S_{H_{n}}$. 

The construction of an expansion set for $H_{n}$ largely follows the example of Thompson's group $V$.
We consider all pairs $(f,D)$, where $D$ is either a singleton or a ray $[k,\infty) \subseteq \mathbb{N}_{i}$ (for appropriate $k$ and $i$), and $f$ is a function with domain $D$ that is expressible as a finite disjoint union of members of $S_{H_{n}}$. We write $(f_{1}, D_{1}) \sim (f_{2}, D_{2})$ if there is some
$s \in S_{H_{n}}$ such that $s$ maps $D_{1}$ bijectively to $D_{2}$, and $f_{2} \circ s = f_{1}$. The relation $\sim$ is an equivalence relation. The equivalence classes are denoted $[f,D]$, and the set of all such equivalence classes is $\mathcal{B}$. This is the required expansion set.

For a pair $b = [f,D] \in \mathcal{B}$, we define $supp(b) = f(D)$. We define $\mathcal{E}(b)$ in one of two ways, depending on whether $D$ is  a singleton or a ray:
\begin{itemize}
\item if $D = \{ x \}$, then $\mathcal{E}(b) = \{ \{ b \} \}$;
\item if $D = [k, \infty) \subseteq \mathbb{N}_{i}$, then we let 
$\mathcal{E}(b) = \{ \{ b \}, \{ [f_{\mid}, \{ k \}], [f_{\mid}, [k+1,\infty)] \} \}$.
\end{itemize}
The result is a simple expansion set $\mathcal{B}$, since $|\mathcal{E}(b)| \leq 2$ for all $b \in \mathcal{B}$. 

Most of the remaining details are reminiscent of the Thompson's group $V$ case. The directed set condition
 follows from the fact that each vertex $v$ in $\Delta^{f}_{\mathcal{B}}$ can be connected by an ascending path
to a vertex $v'$ that represents a partition of $X$. Since any two such partitions have a common refinement, the directed set condition follows as before.

It is straightforward to argue that a given $v \in \mathcal{V}_{\mathcal{B}}$ is the contracting basin of at most one
$b \in \mathcal{B}$. Local finiteness of $\Delta^{f}_{\mathcal{B}}$ follows from Proposition 
\ref{proposition:localfin}. 

The proof that $H_{n}$ acts properly and isometrically on $\Delta^{f}_{\mathcal{B}}$ has appeared elsewhere \cite{Farley, FH2}, and is, in any case, comparatively straightforward.

\bibliographystyle{plain}
\bibliography{biblio}

\end{document}